\documentclass{article}
\usepackage[utf8]{inputenc}
 
\usepackage{float}
\usepackage{color}
\usepackage{graphicx}
\usepackage[dvips]{epsfig}

\usepackage{graphics}
\usepackage{amsbsy, amsmath,amsfonts,amssymb,amsthm,amscd,amssymb,latexsym,xcolor}

\numberwithin{equation}{section}

  \newtheorem{theorem}{Theorem}[section]
  \newtheorem*{conjecture}{Conjecture}
\newtheorem{proposition}{Proposition}[section]

\theoremstyle{remark}
\newtheorem{remark}{Remark}[section]

 \usepackage{chngcntr}
\counterwithin{figure}{section}

\usepackage{hyperref}
\hypersetup{pdfnewwindow=true,colorlinks=true,linkcolor=blue,citecolor=blue,
	filecolor=magenta,urlcolor=blue 
	}
\title{On $O(p)\times O(q)$-invariant constant mean curvature hypersurfaces with singularity}

\author{Hil\'ario Alencar, Ronaldo Garcia, and Greg\'orio Silva Neto}

\begin{document}

\maketitle
\begin{abstract}

We classify the $O(p)\times O(q)$-invariant constant mean curvature hypersurfaces   with singularity at the origin, solving a conjecture of Wu-yi Hsiang.

    
\end{abstract}

\section{Introduction}

In 1982, Wu-Yi Hsiang \cite{hsiang} started the systematic study of constant mean curvature hypersurfaces in Euclidean space $\mathbb{R}^{p+q}$, invariant by the action of the group $O(p)\times O(q),$ $p,q\geq2.$ Using techniques of equivariant geometry, Hsiang proved that the profile curves $\gamma(t)=(x(t),y(t))$ generating constant mean curvature $H$ hypersurfaces of $\mathbb{R}^{p+q}$ invariant by the action of the group $O(p)\times O(q),$ lies in the region $Q_1=\{(x,y):x>0,y>0\}\subset\mathbb{R}^2$ and satisfy the ordinary differential equation
\begin{equation}\label{eq:mean-H0}
\aligned
(p+q-1)H &= \frac{x'(t)y''(t)-x''(t)y'(t)}{(W(t))^3}+(q-1)\frac{y'(t)}{x(t)W(t)}\\
&\quad- (p-1)\frac{x'(t)}{y(t)W(t)},
\endaligned
\end{equation}
where $W(t)=\sqrt{(x'(t))^2+(y'(t))^2}$ is the arc length of $\gamma.$ A solution $\gamma(t)=(x(t),y(t))$ of \eqref{eq:mean-H0} is said a global solution curve if $-\infty<t<\infty$ and $\gamma$ is infinitely extendable in both directions.

Analyzing systematically the extended solutions of \eqref{eq:mean-H0}, Hsiang proved that

\begin{itemize}
    \item[(i)] There are two straight lines $x=\frac{q-1}{H(p+q-1)}$ and $y=\frac{p-1}{H(p+q-1)},$ whose inverse images are cylinders of type $\mathbb{R}^p\times\mathbb{S}^{q-1}$ and $\mathbb{S}^{p-1}\times\mathbb{R}^{q},$ respectively.

    \item[(ii)] Any given global solution of \eqref{eq:mean-H0} is asymptotic to the lines $x=\frac{q-1}{H(p+q-1)}$ when $t\to -\infty$ and $y=\frac{p-1}{H(p+q-1)},$ when $t\to \infty.$ 

    \item[(iii)] Each global solution of \eqref{eq:mean-H0} can have at most one cusp point (singular point) on each axis $x=0$ or $y=0$. Therefore, Hsiang classified  the global solution curves of \eqref{eq:mean-H0} in the following types, namely
\begin{itemize}
    \item[Type A.] With no cusp point.
    \item[Type B.] With exactly one cusp point on the $x$-axis.
    \item[Type C.] With exactly one cusp point on the $y$-axis.
    \item[Type D.] With exactly two cusps (which must be one on each axis).
    \item[Type E.] With exactly one cusp at the origin.
\end{itemize}
\end{itemize}

A local view of the extended singular solutions asymptotic to the coordinate axes are as shown in Fig. \ref{fig:singular_axis}.

 \begin{figure}[H]
	\begin{center}
  \def\svgwidth{1.0\textwidth}
 	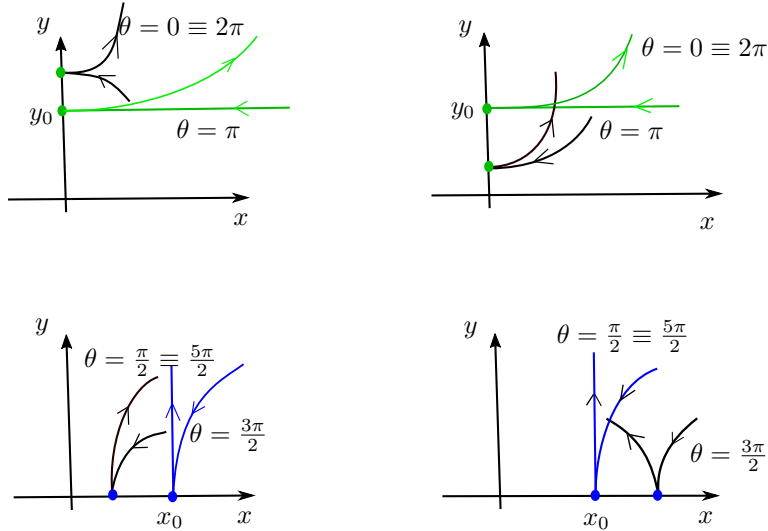
 	 	\vskip .5cm
		\caption{Singular extended solutions asymptotic to the coordinate axes.   }
		\label{fig:singular_axis}
	\end{center}
\end{figure}

About the global solutions with exactly one cusp at the origin (i.e., Type E solutions), Hsiang conjectured that (see \cite[page 353]{hsiang})

\begin{conjecture}\label{conj-Hsiang}
There is only one global solution curve of \eqref{eq:mean-H0} with one cusp point at the origin. i.e., there is only one constant mean curvature $H$ hypersurface of $\mathbb{R}^{p+q},$ $p,q\geq2,$ invariant by the action of the group $O(p)\times O(q)$ and with singularity at the origin.
\end{conjecture}

The purpose of this paper is to prove this conjecture. Namely, we prove

\begin{theorem}\label{thm:main-1}
    There is only one hypersurface of $\, \mathbb{R}^{p+q}$ invariant by $O(p)\times O(q),$ with constant mean curvature $H\ne 0$, whose generating curve is a global solution of \eqref{eq:mean-H0} with one cusp point at the origin. Moreover,
    \begin{itemize}
        \item[(i)] the topological type of the hypersurface, which is singular at the origin, is $(\mathbb{R}\setminus\{0\})\times\mathbb{S}^{p-1}\times\mathbb{S}^{q-1};$
        \item[(ii)] the profile curve is given by $W^s\cup W^u\cup\{0\},$ where the two branches $W^u$ and $W^s$ of the curve are smooth and tangent to the line $y=(\tan\alpha_0) x$ at the origin (see Figure \ref{fig:main-theorem}), where  
 \[
 \tan\alpha_0=\sqrt{ \frac{p-1}{q-1}}
 \]
 and the curvatures of $W^u$ and $W^s$ at the origin are given, respectively, by
 \[
 k_u=\frac{2 H(p + q - 1)}{ 3p + 3q - 4}
 \]
and $k_s=-k_u<0.$
    \end{itemize}
\end{theorem}

\begin{figure}[H]
	\begin{center}
		\def\svgwidth{0.9\textwidth} 
		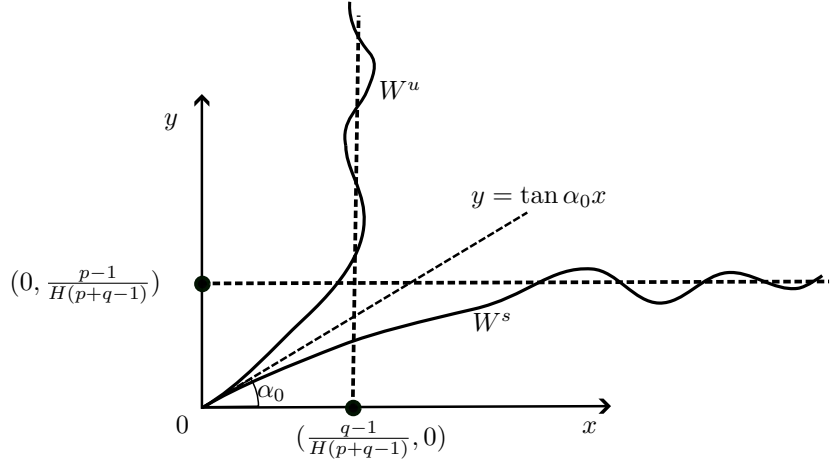
		\caption{Sketch of the generating curve of the   solution with one cusp point (singular point) at the origin.}\label{fig:main-theorem}
	\end{center}
\end{figure}


If the hypersurface is minimal, i.e., $H=0,$ we have
\begin{theorem}\label{thm:main-2}
 The only minimal hypersurface of $\mathbb{R}^{p+q},$ invariant by $O(p)\times O(q),$ with one cusp point at the origin, is the cone given by
 \[ C=\left\{ (U,V)\in \mathbb{R}^p\times\mathbb{R}^q ;(q-1)|V|^2=(p-1)|U|^2\right\}.\]
Moreover, this minimal cone has profile curve $y=(\tan\alpha_0)x=\sqrt{\frac{p-1}{q-1}}x$ and principal curvatures given by
 \[ \lambda_1(t)=0,\;\; \lambda_2(t)=\frac{1}{t}\sqrt{\frac{p-1}{q-1}}, \;\; \lambda_3(t)=- \frac{1}{t}\sqrt{\frac{q-1}{p-1}}. \]
\end{theorem}

\begin{remark}
    In the case $p=q,$ Theorem \ref{thm:main-2} was proved by the first author in Theorem 4.1 of \cite{alencar}.
\end{remark}

\begin{remark} In the proof of Theorem \ref{thm:main-1} we use blowing up techniques for degenerate singularities and invariant manifold theory for a tridimensional system of ordinary differential equations. We remark that the noninvariance of the constant mean curvature equation by homotheties prevents us from using the method developed by E. Bombieri, E. De Giorgi, and E. Giusti in \cite{bgg} to classify minimal hypersurfaces in the Euclidean space, and transform the constant mean curvature equation in a bidimensional system of ordinary differential equations. This forces us to analyze a tridimensional system of ordinary differential equations with lines of singularities (normallly hyperbolic) and degenerated singularities.
\end{remark}

 \section{Proof of Theorem \ref{thm:main-1}}

Let $G=O(p)\times O(q)$ acting in the standard way on $\mathbb{R}^{p}\times\mathbb{R}^{q}=\mathbb{R}^{p+q}.$ The orbit of this action which passes through the point $(U,V)\in\mathbb{R}^p\times \mathbb{R}^q$ is given by $\mathbb{S}^{p-1}(|U|)\times\mathbb{S}^{q-1}(|V|),$ and the orbit space $\mathbb{R}^{p+q}/G$
can be represented by 
\begin{equation}
Q:=\pi(\mathbb{R}^{p+q})=\{(x,y)\in\mathbb{R}^2; x\geq 0,\ y\geq 0\},
\end{equation}
where $\pi(U,V)=(|U|,|V|).$

If $\varphi:\Sigma\to\mathbb{R}^{p+q}$ is a hypersurface invariant by the group of rotation  $O(p)\times O(q),$ then $\varphi(\Sigma)$ has the principal curvatures, 
\begin{equation}\label{eq:principal_curvatures}
\lambda_1(t)=\frac{x'(t)y''(t)-x''(t)y'(t)}{(W(t))^3},\; \lambda_2(t)=\frac{y'(t)}{x(t)W(t)}, \; \lambda_3(t)=-\frac{x'(t)}{y(t)W(t)},
\end{equation}
where $W(t)=\sqrt{(x'(t))^2+(y'(t))^2},$ $\lambda_1$ has multiplicity one, $\lambda_2$ has multiplicity $q-1,$ and $\lambda_3$ has multiplicity $p-1.$ Thus, the normalized mean curvature $H$ is given by
\begin{equation}\label{eq:mean-H}
\aligned
(p+q-1)H &=  \lambda_1+(q-1)\lambda_2+(p-1)\lambda_3\\
&= \frac{x'(t)y''(t)-x''(t)y'(t)}{(W(t))^3}+(q-1)\frac{y'(t)}{x(t)W(t)}\\
&\quad- (p-1)\frac{x'(t)}{y(t)W(t)}.
\endaligned
\end{equation}
This differential equation is homogeneous and so the solutions are  invariant by reparametrizations.
If we assume that the profile curve is parametrized by the arc length, i.e., $W(t)=1$, then there exists an angle function $\theta(t)$ such that 
\begin{equation}\label{angle-function}
x'(t)=\cos(\theta(t)) \quad \mbox{and} \quad y'(t)=\sin(\theta(t)).
\end{equation}
Replacing \eqref{angle-function} in \eqref{eq:mean-H}, we obtain 
\begin{equation}\label{eq:H-theta}
\theta'(t)=(p+q-1)H+(p-1)\frac{\cos(\theta(t))}{y(t)}-(q-1)\frac{\sin(\theta(t))}{x(t)}.
\end{equation}
Thus, combining equations \eqref{angle-function} and \eqref{eq:H-theta}, we conclude that the profile curves $\gamma(t)=(x(t),y(t))$ generating constant mean curvature $H$ hypersurfaces of $\mathbb{R}^{p+q},$ invariant by the action of the group $O(p)\times O(q),$ are the projection in the $xy$-plane of the solutions of the system of first order ordinary differential equations
\begin{equation}\label{eq:H}
\left\{
\aligned
x'(t)&= \cos (\theta(t))\\
y'(t)&=\sin(\theta(t))\\
    \theta'(t)& =(p+q-1)H+(p-1)\frac{\cos(\theta(t))}{y(t)}-(q-1)\frac{\sin(\theta(t))}{x(t)},
    \endaligned\right.
\end{equation}
where $(x,y)\in Q_1=\{(x,y); x> 0\; \mbox{and} \;y > 0\}$ and $\theta\in [0,2\pi)$.

The solutions of the system of differential equations \eqref{eq:H} are the integral curves of the vector field
   \begin{equation} \label{eq:fieldX} 
   X(x,y,\theta)=\left(\cos\theta,\sin\theta, (p+q-1)H+(p-1)\frac{\cos\theta}{y}-(q-1)\frac{\sin\theta}{x}\right).
   \end{equation}
   
A regular extension of $X$ is given by the vector field $Y=xy X,$ $x>0,y>0$. The orbits of $Y$ are the same of $X$ in the region $Q_1\times\mathbb{R}$.



We need to analyze the behavior of $Y$ in the neighborhood of the axis $\theta$ with $x\geq 0$ and $y\geq 0$. First we observe that
$Y(0,0,\theta)=(0,0,0)$ and the three eigenvalues of $DY(0,0,\theta)$ are zeros, i.e., a very degenerated  set of singularities.  

In order to analyze this line of singular points it will considered a cylindrical blowing-up.

It consists of replacing the  vector field $Y$   by a vector field $\tilde Y$ defined in $\mathbb{S}^1\times \mathbb{R}$ which has  has less degenerate singularities. For details about the technique of blowing up see, for example, \cite{dumortier}   and \cite{takens}.

\begin{proposition}\label{prop:Y-tilde}
Consider the cylindrical blowing up  along the axis $\theta$ 
\[(x,y,\theta)=(r\cos\alpha,  r\sin\alpha,  \theta)=R(r,\alpha,\theta),\] with $0\leq \alpha\leq \pi/2$. Taking the pullback of $ Y$ by $R $ and dividing by $r$, we obtain the vector field
\[
\Tilde{Y}:=\frac{1}{r}R_{*}Y=:(Y_r,Y_\alpha,Y_\theta),
\]
where
\begin{equation}\label{eq:tilde-Y}
\left\{    \begin{aligned}
Y_r=r'&=r\sin\alpha\cos\alpha\cos(\theta-\alpha)\\
Y_\alpha=\alpha'&=\sin\alpha\cos\alpha\sin(\theta-\alpha)\\
Y_\theta=\theta'&=\frac{r}{2}(p+q-1)H\sin 2\alpha + (p-1)\cos\alpha\cos\theta\\
&\quad- (q-1)\sin\alpha\sin\theta.\\
\end{aligned}\right.
\end{equation}
\end{proposition}
\begin{proof}

Taking the derivatives in $x=r\cos\alpha$ and $y=r\sin\alpha$ relative to paramater $t,$ we obtain
\[
\left\{\begin{aligned}
x'&=r'\cos\alpha - r\alpha'\sin\alpha\\
y'&=r'\sin\alpha + r\alpha'\cos\alpha,\\
\end{aligned}\right.
\]
which gives,
\[
\left\{\begin{aligned}
\alpha'&=\frac{1}{r}(y'\cos\alpha - x'\sin\alpha)\\
r'&=x'\cos\alpha + y'\sin\alpha.\\
\end{aligned}\right.
\]
Replacing the expressions $x'=\cos\theta$ and $y'=\sin\theta$ of the vector field $X$ on the expressions above, we obtain
\[
\left\{\begin{aligned}
\alpha'&=\frac{1}{r}\sin(\theta-\alpha)\\
r'&=\cos(\theta-\alpha)\\
\theta'&=(p+q-1)H + (p-1)\frac{\cos\theta}{r\sin\alpha}-(q-1)\frac{\sin\theta}{r\cos\alpha}.
\end{aligned}\right.
\]
On the other hand, replacing the expressions 
\[
\left\{\begin{aligned}
x'&=xy\cos\theta=r^2\cos\alpha\sin\alpha\cos\theta,\\
y'&=xy\sin\theta=r^2\cos\alpha\sin\alpha\sin\theta,
\end{aligned}\right.
\] 
of the vector field $Y$, we obtain
\begin{equation}\label{tilde-Y-0}
    \left\{\begin{aligned}
r'&=r^2\sin\alpha\cos\alpha\cos(\theta-\alpha)\\
\alpha'&=r\sin\alpha\cos\alpha\sin(\theta-\alpha)\\
\theta'&=(p+q-1)Hr^2\sin\alpha\cos\alpha + (p-1)r\cos\alpha\cos\theta\\
&\quad - (q-1)r\sin\alpha\sin\theta.\\
\end{aligned}\right.
\end{equation}
Since $r>0$, 
the vector field $(1/r)Y$ has the same orbits as $Y.$ Thus, by taking the pullback of $ Y$ by $R $ and dividing by $r$, we obtain the vector field
\[
\Tilde{Y}:=\frac{1}{r}R_{*}Y=:(Y_r,Y_\alpha,Y_\theta),
\]
as stated in equation \eqref{eq:tilde-Y}.
\end{proof}

The divisor of $R,$ i.e., the singular set of $R,$ is given by the set where $r=0$. From the  equation \eqref{eq:tilde-Y}, this set is invariant by the flow of  $\tilde Y$. Thus, we can define the bimensional vector field $\tilde Y_0$ by 
\[
\tilde Y_0(\alpha,\theta)=\pi_1(\tilde Y(0,\alpha,\theta)),
\]
where $\pi_1(x,y,z)=(y,z),$ i.e., $\tilde Y_0$ is the vector field associated to the bidimensional system of equations
\begin{equation}\label{tilde-Y0}
    \left\{\begin{aligned}
\alpha'&= \sin\alpha\cos\alpha\sin(\theta-\alpha)\\
\theta'&=  (p-1)\cos\alpha\cos\theta - (q-1)\sin\alpha\sin\theta.\\
\end{aligned}\right.
\end{equation}
The first step to understand the phase portrait of $\tilde Y$  is to analyze the phase portrait of $\tilde Y_0.$

Also we observe that in the coordinates $ ( \alpha,\theta)$ the region $  [0,\pi/2]\times\mathbb{R}$
  is invariant by the flow of $\tilde Y_0$, which is $2\pi-$periodic in $\alpha$ and $\theta$.  
 
\begin{proposition}\label{prop:6singular_points}
   In the rectangle $R=[0,\pi/2]\times[0,2\pi) $  the vector field 
$\tilde Y_0$ defined by equation \eqref{tilde-Y0} has six singular points, all of them hyperbolic (saddles, foci or nodes), namely $p_1=(0,\pi/2)$, $p_2=(0,3\pi/2)$,
   $p_3=(\pi/2,0)$, $p_4=(\pi/2,\pi)$, $p_5=(\alpha_0,\alpha_0)$, $p_6=(\alpha_0,\pi+\alpha_0)$,  where 
    \begin{equation}\label{alpha-0}    
     \alpha_0=\arctan\left( \sqrt{\frac{p-1}{q-1}}\right)=\arccos{\left( \sqrt{\frac{q-1}{p+q-2}}\right)}.
     \end{equation}
    The phase portrait of $\tilde Y_0$ in the region $[0, \pi/2]\times [0,2\pi ]$ is as shown in Figure \ref{fig:phase_X0a}.
\end{proposition}

%
  %

    \begin{figure}[ht]
	\begin{center}
 \def\svgwidth{0.8\textwidth} 
	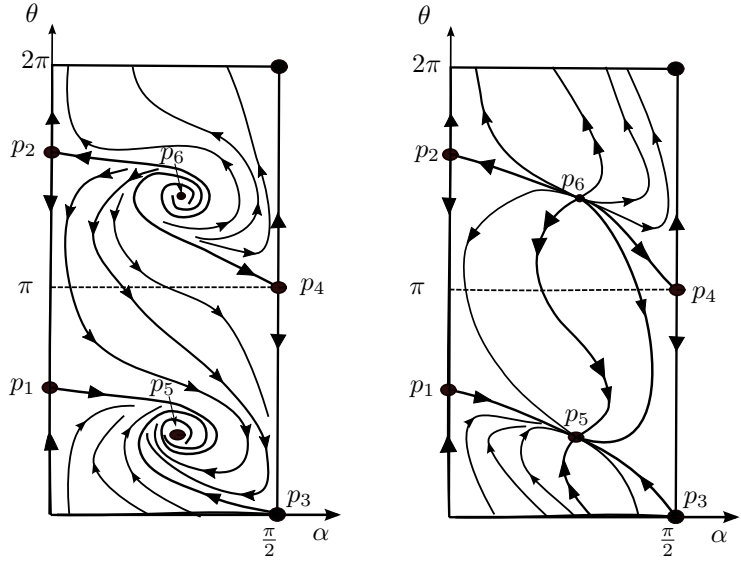
 	 	 
		 \caption{Integral curves of $\tilde Y_0$ in the region
  $[0,\pi/2]\times [0,2\pi]$.
  Left: The singular point $p_5$ (resp. $p_6$) is an attractor (resp. repeller). Here $p+q<8$.   Both are foci. 
  Right: The singular point $p_5$ (resp. $p_6$) is an attractor (resp. repeller).  Here $p+q\geq 8$. Both are nodes.  }
		 	\label{fig:phase_X0a}
	\end{center}
\end{figure} 
 \begin{proof}

By the first coordinate of $\tilde Y_0,$ the singularities will occur in the points $(\alpha,\theta)\in[0,\pi/2]\times[0,2\pi)$ such that $\sin\alpha=0,$ $\cos\alpha=0$ or $\sin(\theta-\alpha)=0.$ This gives, respectively, $\alpha=0,$ $\alpha=\pi/2$ and $\theta=\alpha+k\pi,$ $k=0,1.$ 

When $\alpha=0$,  $\tilde Y_0(0,\theta)=(0,(p-1)\cos\theta)$. Therefore $p_1=(0,\pi/2)$ and $p_2=(0,3\pi/2)$ are singular points. For $\alpha=\pi/2$, $\tilde Y_0(\pi/2,\theta)=(0,(q-1)\sin\theta)$. Therefore $p_3=(\pi/2,0)$ and $p_4=(\pi/2,\pi)$ are singular points. Finally, for $\alpha\in (0,\pi/2)$ the singular points of $\tilde Y_0$ are defined by
 $\theta=\alpha$ or $\theta-\alpha=\pi$ and $(p-1)\cos^2\alpha=(q-1)\sin^2\alpha$. It follows that
 \[\tan\alpha=\pm \sqrt{\frac{p-1}{q-1}}.\]
 This gives $p_5=(\alpha_0,\alpha_0)$ and $p_6=(\alpha_0,\alpha_0+\pi)$ in the rectangle $[0,\pi/2]\times[0,2\pi),$ where $\alpha_0$ is given by Equation \eqref{alpha-0}.

 We claim that all singular points are hyperbolic. This follows analyzing    $D\tilde Y_0 (p_i),$ $i=1,\ldots,6$. In fact, since
\[
\frac{\partial \tilde Y_0}{\partial \alpha}(\alpha,\theta)=\left[
\begin{array}{c}
     \cos(2\alpha)\sin(\theta-\alpha)-\sin\alpha\cos\alpha\cos(\theta-\alpha) \\
     -(p-1)\sin\alpha\cos\theta - (q-1)\cos\alpha\sin\theta\\ 
\end{array}
\right]
\]
and
\[
\frac{\partial \tilde Y_0}{\partial \theta}(\alpha,\theta)=\left[
\begin{array}{c}
     \sin\alpha\cos\alpha\cos(\theta-\alpha)\\
     -(p-1)\cos\alpha\sin\theta-(q-1)\sin\alpha\cos\theta\\ 
\end{array}
\right],
\]
we have
\[
D\tilde Y_0(p_1)=-D\tilde Y_0(p_2)=\left[
\begin{array}{cc}
   1  & 0 \\
   -(q-1)  & -(p-1)\\
\end{array}
\right]
\]
and
\[
D\tilde Y_0(p_3)=-D\tilde Y_0(p_4)=\left[
\begin{array}{cc}
   1  & 0 \\
   -(p-1)  & -(q-1)\\
\end{array}
\right].
\]
This gives clearly the eigenvalues and proves that $p_1,$ $p_2,$ $p_3,$ and $p_4$ are hyperbolic of saddle type. On the other hand, denoting by 
\[
\lambda_0=\sin\alpha_0\cos\alpha_0=\frac{\sqrt{(p-1)(q-1)}}{p+q-2}>0,
\]
we have
\[
D\tilde Y_0(p_5)=-D\tilde Y_0(p_6)=\lambda_0\left[
\begin{array}{cc}
   -1  & 1 \\
   -(p+q-2)  & -(p+q-2)\\
\end{array}
\right].
\]
This gives the eigenvalues 
\[
\mu_a=\lambda_0\frac{-(p+q-1)+\sqrt{(p+q-1)^2-8(p+q-2)}}{2}
\]
and 
\[
\mu_b=\lambda_0\frac{-(p+q-1)-\sqrt{(p+q-1)^2-8(p+q-2)}}{2}
\]
of $D\tilde Y_0(p_5),$ and the eigenvalues 
\[
\nu_a=-\mu_a\quad \mbox{and} \quad\nu_b=-\mu_b
\] 
of $D\tilde Y_0(p_6).$ Clearly, $\mu_a$ and $\mu_b$ have negative real part and thus $p_5$ is an attractor. Analogously, $p_6$ is a repeller. Moreover, by the sign of the discriminant 
\[
D=(p+q-1)^2-8(p+q-2)=(p+q)^2-10(p+q)+17,
\] 
we have that both are foci for $p+q<5+2\sqrt{2}<8$ (i.e., complex eigenvalues) and both are nodes for $p+q\geq 8$ (i.e., real eigenvalues). 

In order do complete the phase portrait, notice also that the lines $\alpha=0$ and $\alpha=\pi/2$ are both invariant by $\tilde Y_0.$

Using the description of the local phase portrait near the hyperbolic singular points we can globalize the analysis observing the behavior of the invariant separatrices (stables and unstables). A graphic analysis shows that, in the rectangle $R=[0,\pi/2]\times[0,2\pi),$ we have
\begin{itemize}
\item[(i)] $W^s(0,\pi/2)\cap R$ and $W^u(0,3\pi/2)\cap R$ are contained in $\alpha=0;$
\item[(ii)] $W^u(\pi/2,\pi)\cap R $ and $W^u(0,3\pi/2)\cap R$ are contained in $\alpha=\pi/2;$
\item[(iii)] the $\omega- $limit  set of $W^u(\pi/2,0)\cap R$ and $W^u(0,\pi/2)\cap R$ is the attractor $\{p_5\};$
\item[(iv)]  the $\alpha$-limit set of $W^u(\pi/2,\pi)\cap R$ and $W^u(0,3\pi/2)\cap R$ is the repeller $\{p_6\};$
\item[(v)]  the vector field $\tilde Y_0$ is transversal to the line $\theta=\pi$ and for any point $p$ in the open interval $(\alpha,\pi)$, $0<\alpha<\pi$,  we have that $\omega(p)=\{p_5\}$ and $\alpha(p)=\{p_6\}$.
\end{itemize}

Gluing the local phase portraits near the hyperbolic singularities and taking into account  the properties i), ii), iii), iv) and v)  listed above we obtain the result stated and illustrated in Figure \ref{fig:phase_X0a}.
 \end{proof}

\begin{proposition}\label{prop:WuP5}
      The point  $P_5=(0, \alpha_0,\alpha_0)$ is a singular  hyperbolic saddle of the vector field $\tilde Y$ given by equation \eqref{eq:tilde-Y}. The stable manifold  $W^s(P_5)$ is contained in the plane $r=0,$ the unstable manifold $W^u(P_5)$ is one dimensional, and it is locally parametrized by
\[
\left\{
\begin{aligned}
     \theta(r)=& \alpha_0 + k\pi +k_1r+k_2r^2+k_3r^3+O(r^4)\\
     \alpha(r)=& \alpha_0 +l_1 r+l_2r^2+l_3 r^3+O(r^4),
 \end{aligned}\right.
\]
 where
  \begin{equation}\label{eq:WuP5}
     \aligned
     l_1=&\frac{ H(p + q - 1)}{3(p + q) - 4},\;\; \;\;k_1=2l_1,\;\; k_2=3l_2,\\
     l_2=& -\frac{H^2 (p + q - 2) (p - q) (p + q - 1)^2}{ 2 (2 p + 2 q - 1) (3 p + 3 q - 4)^2\sqrt{(p - 1)(q - 1)} }.
     \endaligned
 \end{equation}
      
\end{proposition}
 \begin{proof}

The eigenvalues of $D\tilde Y(P_5)$ are

\begin{align*}
\mu_1=& \frac{\sqrt{(q - 1) (p - 1)} }{ p + q - 2} >0,\\
   \mu_2=&\frac{\sqrt{(q - 1) (p - 1)}(\sqrt{ D}-p-q+1)}{ 2(p + q - 2)}, \\
 \mu_3=&-\frac{\sqrt{(q - 1) (p - 1)}(\sqrt{ D}   + p + q - 1)}{ 2(p + q - 2)},  \\
\end{align*} 
where 
\[
D=(p+q)^2-10(p+q) +17.
\] 
Since
\[
\mu_2\mu_3= 2(p - 1)(q - 1)/(p + q - 2)>0,
\]  
if $D>0$, then both eigenvalues, $\mu_2$ and $\mu_3$  are negative. Otherwise, i.e., when $D<0$, both eigenvalues are complex conjugated and with negative real part.
 
Using Proposition \ref{prop:6singular_points}
   we have that singular point $p_5$   is a hyperbolic attractor of $\tilde Y_0$. It is a focus for $p+q<5+2\sqrt{2}<8$ and a node   for $p+q\geq 8$. 

By Invariant Manifold Theory, see \cite{fe}, \cite{hps} and \cite[Chapter 2]{pame}, the unstable manifold $W^u(P_5)$   is  smooth  and   one dimensional.

    \begin{figure}[H]
	\begin{center}
 	\def\svgwidth{0.9\textwidth}
 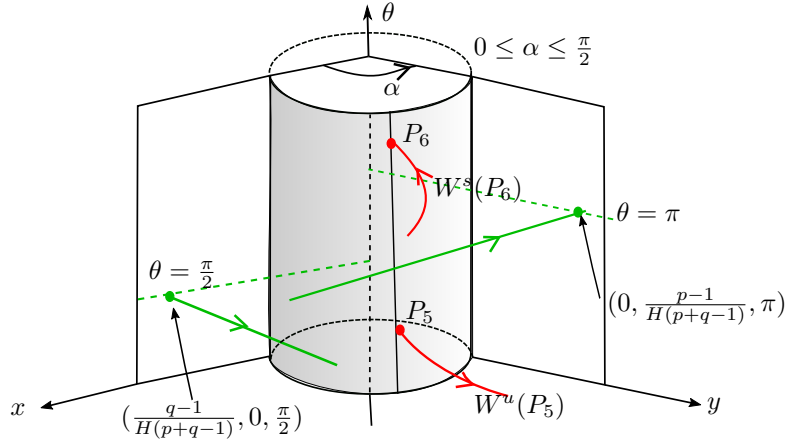
 	 	%
		\caption{ Blowing-up and the local invariant manifolds $W^u(P_5)$ and $W^s(P_6)$.}
		\label{fig:WsWuP5P6}
	\end{center}
\end{figure} 

The projections of the invariant manifolds $W^u(P_5)$ and 
$W^s(P_6)$ in a neighborhood of $0$ are as shown in Fig. \ref{fig:projWsWu} below.
    \begin{figure}[H]
	\begin{center}
 	\def\svgwidth{0.6\textwidth}
\begingroup%
  \makeatletter%
  \providecommand\color[2][]{%
    \errmessage{(Inkscape) Color is used for the text in Inkscape, but the package 'color.sty' is not loaded}%
    \renewcommand\color[2][]{}%
  }%
  \providecommand\transparent[1]{%
    \errmessage{(Inkscape) Transparency is used (non-zero) for the text in Inkscape, but the package 'transparent.sty' is not loaded}%
    \renewcommand\transparent[1]{}%
  }%
  \providecommand\rotatebox[2]{#2}%
  \newcommand*\fsize{\dimexpr\f@size pt\relax}%
  \newcommand*\lineheight[1]{\fontsize{\fsize}{#1\fsize}\selectfont}%
  \ifx\svgwidth\undefined%
    \setlength{\unitlength}{261.50663841bp}%
    \ifx\svgscale\undefined%
      \relax%
    \else%
      \setlength{\unitlength}{\unitlength * \real{\svgscale}}%
    \fi%
  \else%
    \setlength{\unitlength}{\svgwidth}%
  \fi%
  \global\let\svgwidth\undefined%
  \global\let\svgscale\undefined%
  \makeatother%
  \begin{picture}(1,0.78302913)%
    \lineheight{1}%
    \setlength\tabcolsep{0pt}%
    \put(0,0){\includegraphics[width=\unitlength,page=1]{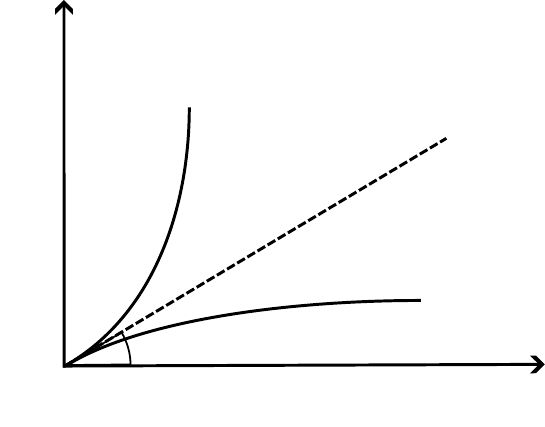}}%
    \put(0.04769787,0.74188898){\color[rgb]{0,0,0}\makebox(0,0)[lt]{\lineheight{1.25}\smash{\begin{tabular}[t]{l}$y$\end{tabular}}}}%
    \put(0.96002549,0.04929666){\color[rgb]{0,0,0}\makebox(0,0)[lt]{\lineheight{1.25}\smash{\begin{tabular}[t]{l}$x$\end{tabular}}}}%
    \put(0.2738447,0.13909956){\color[rgb]{0,0,0}\makebox(0,0)[lt]{\lineheight{1.25}\smash{\begin{tabular}[t]{l}$\alpha_0$\end{tabular}}}}%
    \put(0.38135886,0.60293318){\color[rgb]{0,0,0}\makebox(0,0)[lt]{\lineheight{1.25}\smash{\begin{tabular}[t]{l}$W^u$\end{tabular}}}}%
    \put(0.77967175,0.27161393){\color[rgb]{0,0,0}\makebox(0,0)[lt]{\lineheight{1.25}\smash{\begin{tabular}[t]{l}$W^s$\end{tabular}}}}%
    \put(0.88464855,0.44662805){\color[rgb]{0,0,0}\makebox(0,0)[t]{\lineheight{1.25}\smash{\begin{tabular}[t]{c}$y=(\tan\alpha_0)x$\end{tabular}}}}%
    \put(0.09720857,0.0744365){\color[rgb]{0,0,0}\makebox(0,0)[t]{\lineheight{1.25}\smash{\begin{tabular}[t]{c}$0$\end{tabular}}}}%
  \end{picture}%
\endgroup%

		\caption{ Projections of the  invariant manifolds $W^u(P_5)$ and $W^s(P_6)$. }
		\label{fig:projWsWu}
	\end{center}
\end{figure}
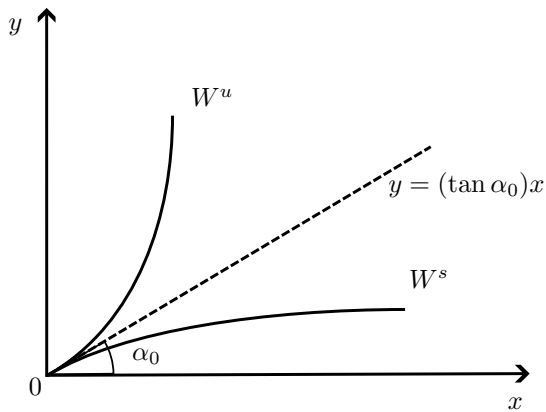 

 In order to   compute  the smooth invariant manifolds $W^u(P_5)$  we express  this regular curve  by Taylor's  series 
 \begin{align*}
     \theta(r)=& \alpha_0  +k_1r+k_2r^2+k_3r^3+O(r^4)\\
     \alpha(r)=& \alpha_0 +l_1 r+l_2r^2+l_3 r^3+O(r^4).
 \end{align*}
   
 The coefficients $(k_i,l_i),$ $i=1,2,3,$ will be obtained using the properties of invariant manifolds.
   
 For completeness, it will be included here the explicit calculations.

 First we write  the vector field $\dot r=Y_r, \dot \alpha=Y_\alpha, \dot \theta=Y_\theta$ given in the system of equations \eqref{eq:tilde-Y} as a system of  differential equation (elimination of the variable $t$) to obtain the two differential 1-forms $\omega_1$ and $\omega_2$:
\[
\left\{
 \begin{aligned}
     \omega_1&(r,\alpha,\theta,dr,d\alpha,d\theta)= Y_\alpha (r,\alpha,\theta)dr- Y_r (r,\alpha,\theta)d\alpha=0\\
     \omega_2&(r,\alpha,\theta,dr,d\alpha,d\theta)= Y_\theta (r,\alpha,\theta)dr- Y_r(r,\alpha,\theta)d\theta=0.
 \end{aligned}\right.
\]
 The kernel of $(\omega_1,\omega_2)$ is generated by $\tilde Y$ and so the solutions of $\omega_1=\omega_2=0$ are the orbits of the vector field $\tilde Y$.
 
Evaluating $ \omega_i(\theta(r),\alpha(r),dr,d\alpha) =0,$ for $i=1,2$, with 
\[
d\theta= (k_1+ 2k_2r+3k_3r^2+\cdots)dr
\] 
and 
\[
d\alpha=(l_1+ 2l_2r+3l_3r^2+\cdots)dr,
\]
and using the method of formal power series the coefficients $l_i$ and $k_i$ can be effectively computed.

In fact, from the equation above and the system of equations \eqref{eq:tilde-Y} it follows, after simplifications,
\[
\begin{aligned}
    \omega_1(r,\alpha(r),\theta(r),dr,d\theta)&= Y_\alpha(r,\alpha(r),\theta(r))dr\\
    &\quad- Y_r(r,\alpha(r),\theta(r))(l_1+ 2l_2r+3l_3r^2+\cdots)dr\\
    &=-rdr [ (k_1 - 2l_1)\sin(\alpha_0)\cos(\alpha_0)+O(r)]\\
    \omega_2(r,\alpha(r),\theta(r),dr,d\theta)&= Y_\theta(r,\alpha(r),\theta(r))dr\\
    &\quad- Y_r(r,\alpha(r),\theta(r))(k_1+ 2k_2r+3k_3r^2+\cdots)dr\\
    &=- rdr \left[(H-k_1)(q + p - 1)\right.\\
    &\quad\left.- (p+q - 2)l_1)\frac{\sqrt{(p - 1)(q - 1)}}{p+q-2}  +O(r)]\right. .
\end{aligned}
\]
Solving the linear system
\begin{align*}
  \left\{  \begin{aligned} &k_1 - 2l_1=0\\
&(H-k_1)(q + p - 1) - (p+q - 2)l_1=0,
\end{aligned}
  \right.
\end{align*}
it follows
 \[ 
 l_1= \frac{ H(p + q - 1)}{3(p + q) - 4},\;\; \;\;k_1=2l_1.
 \]
 Analogously,   developing $\omega_1$ and $\omega_2,$ up to order $r^2,$ it follows that
 \begin{equation*}
     \aligned
    %
     l_2=& -\frac{H^2 (p + q - 2) (p - q) (p + q - 1)^2}{ 2 (2 p + 2 q - 1) (3 p + 3 q - 4)^2\sqrt{(p - 1)(q - 1)} },\;\; k_2=3l_2.
     \endaligned
 \end{equation*}
With this method we can obtain the Taylor's series of $\alpha(r)$ and $\theta(r)$. This ends the proof.
 \end{proof}

 \begin{proposition}\label{prop:WsP6}

      The point  $P_6=(0, \alpha_0,\alpha_0+\pi)$  is a  hyperbolic saddle of the vector field $\tilde Y$ given by the system of equations \eqref{eq:tilde-Y}. Moreover, $W^s(P_6)$ is locally parametrized by
\[
\left\{
\begin{aligned}
     \theta(r)=& \alpha_0 +  \pi +k_1r+k_2r^2+k_3r^3+O(r^4)\\
     \alpha(r)=& \alpha_0 +l_1 r+l_2r^2+l_3 r^3+O(r^4),
 \end{aligned}\right.
 \]
 where
  \begin{equation}
     \aligned
     l_1=&-\frac{ H(p + q - 1)}{3(p + q) - 4},\;\; \;\;k_1=2l_1,\;\; k_2=3l_2,\\
     l_2=&-  \frac{H^2 (p + q - 2) (p - q) (p + q - 1)^2}{ 2 (2 p + 2 q - 1) (3 p + 3 q - 4)^2\sqrt{(p - 1)(q - 1)} }.
     \endaligned
 \end{equation}
      \end{proposition}

      \begin{proof} Similar to that of Proposition \ref{prop:WuP5}.
      
\end{proof} 

\begin{proposition}\label{prop:sing_normally}
    The lines $\ell_i=\{(r,p_i)\} $   $(i=1,\ldots, 4)$ are normally hyperbolic of saddle type of the vector field $\tilde Y$. See Fig. \ref{fig:invariantes_surfaces}.
\end{proposition}
    \begin{figure}[H]
	\begin{center}
 	\def\svgwidth{1.0\textwidth}
 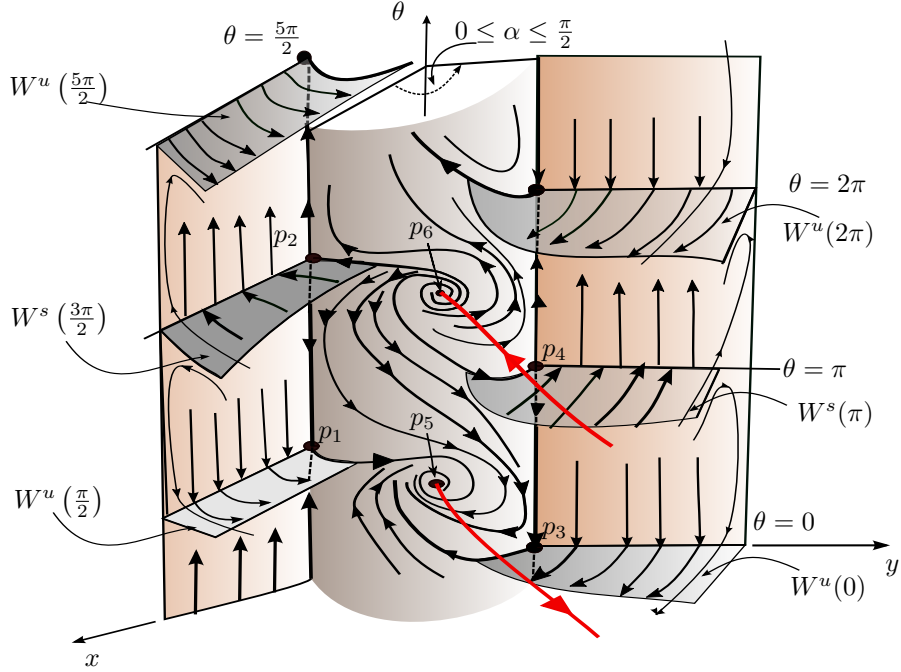
		\caption{ Blowing-up and phase portrait near the cylinder. It is shown the invariant surfaces (stable and unstable) of the normally hyperbolic lines and of the hyperbolic singular points. }
		\label{fig:invariantes_surfaces}
	\end{center}
\end{figure} 
\begin{proof} From the definition of $\tilde Y$, see equation \eqref{eq:tilde-Y}, it follows that $\tilde Y(r,p_i)=0$.
The nonzero eigenvalues of $D\tilde Y(0,p_i)$ are the same of $D\tilde Y_0(p_i)$ $(i=1,\ldots, 4)$ which are hyperbolic saddles.
By continuity and invariant manifold theory it follows that   $W^s(\ell_i)$ and $W^u(\ell_i)$ are invariant two dimensional surfaces, fibered by integral curves of $\tilde Y$ which are asymptotic to singular lines $\ell_i$. Therefore, only $W^u(0)=\pi(W^u(P_5))$ and 
 $W^s(0)=\pi(W^u(P_6))$ are asymptotic to the origin.
\end{proof}

 \begin{proposition}\label{prop:projecao_WsWu}

  The projections of $W^u(P_5)$ and $W^s(P_6)$, see Figure \ref{fig:WsWuP5P6}, are regular curves in the plane $xy$ and they are tangent at the origin $(0,0)$ to the straight line $y=\tan\alpha_0 x$ and they are as shown in Figure \ref{fig:main-theorem}. Moreover, the curvature of $W^u(0)=\pi(W^u(P_5))$ at $(0,0)$ is given by
  \[ k_u(0)=\frac{2H(p + q - 1)}{ 3p + 3q - 4}>0\]
 and the curvature of the planar curve $W^s(0)=\pi(W^s(P_6))$ at $(0,0)$
  is given by  $k_s(0)=-k_u(0)<0$. Here $\pi(x,y,\theta)=(x,y)$.
 \end{proposition}

 \begin{proof}
 The projections of $W^u(P_5)$ and $W^s(P_6)$, see Figure \ref{fig:WsWuP5P6}, in the plane $xy$ are tangent at the origin $(0,0)$ to the straight line $y=\tan\alpha_0 x$ and they are as shown in Figure \ref{fig:main-theorem}.

 In fact, from the cylinder blowing-up we have that
 $x=r\cos\alpha$ and $y=r\sin\alpha$.

 Therefore, the projection of $W^u(P_5)$ is given by
 \[\Gamma^u(r)=(r\cos\alpha(r),r\sin \alpha(r)),\]
where
\[ \alpha(r)=\alpha_0+l_1r+l_2r^2+l_3r^3+\cdots \]
is as given in Proposition \ref{prop:projecao_WsWu}. Then
 \begin{align*}   
 \frac{d}{dr}(\Gamma^u(r))(0)&= (\cos\alpha_0,\sin\alpha_0)=\left(\sqrt{\frac{q - 1}{p + q - 2}} , \sqrt{\frac{p - 1}{p + q - 2}}\right) \\
  \frac{d^2}{dr^2}(\Gamma^u(r))(0)&=    \frac{2 H(p + q - 1)}{3(p + q) - 4} \left(-\sqrt{\frac{p - 1}{p + q - 2}} , \sqrt{\frac{q - 1}{p + q - 2}}
  \right).
 \end{align*}
Therefore, the  curvature of $\Gamma^u(r)$ 
is given by \[k_u=2l_1+ 6l_2r+O(r^2),\] which implies  

\[ k_u(0)=\frac{2H(p + q - 1)}{ 3p + 3q - 4}>0.\]

 
 The curvature of $W^s(0)=\pi(W^s(P_6))$ at $(0,0)$
 can be evaluated similarly, using Proposition \ref{prop:WsP6}.  Performing the calculations it follows that  $k_s=-k_u<0$. 
\end{proof}

Now, we are able to conclude the proof of Theorem \ref{thm:main-1}.

\begin{proof}[Conclusion of the proof of Theorem \ref{thm:main-1}]

By Propositions  \ref{prop:sing_normally} and \ref{prop:projecao_WsWu}  the  projections of $ W^u(P_5)$  and $W^u(P_6)$ in the plane $xy$ are the   unique solutions asymptotic to the origin.
These projections are obtained doing the blowing down of the orbits of the vector field $\tilde Y$ in the neighborhood of the divisor $\{r=0\}$.

By   \cite[Corollary 1, page 347]{hsiang} the curve $W^u=W^u(0)$ is regular and asymptotic to the line 
$x=(q-1)/(p+q-1) $ at infinity and $W^s=W^s(0)$ is asymptotic to the line $y=(p-1)/(p+q-1) $ at infinity. From this analysis, it follows that the unique solutions asymptotic to the origin are as shown in  Fig. \ref{fig:main-theorem}. 
\end{proof}

\section{Proof of Theorem \ref{thm:main-2}}

Theorem \ref{thm:main-2} is almost a immediate consequence of Theorem \ref{thm:main-1}, but we include a proof here since, in this case, the projections of $W^u(P_5)$ and $W^u(P_6)$ are degenerate and this fact was not proven in the last section. Moreover, we can calculate explicitly these projections, obtaining the cones as described in Theorem \ref{thm:main-2}.
 
\begin{proof}[Proof of Theorem \ref{thm:main-2}]

When $H=0$ the integral curves of $Y$ are invariant by homotheties. As in Proposition \ref{prop:Y-tilde} the pullback of $Y$ by the cylindrical blowing up $R$ is given by the vector field $\tilde Y$ given by
   \begin{equation}
       \left\{\aligned
        Y_r&=r\sin\alpha\cos\alpha\cos(\theta-\alpha)\\
       Y_\alpha &= \sin\alpha\cos\alpha \sin(\theta-\alpha)\\
      Y_\theta&= (p-1)\cos\theta\cos\alpha-(q-1)\sin\alpha\sin\theta .
       \endaligned\right.
      \label{eq:blowingupHzero}
   \end{equation}

The unstable manifold $W^u(P_5)$ of $P_5$ and stable manifold $W^s(P_6)$ of $P_6$ are one dimensional and when $H=0$ they are straight lines, see Figure \ref{fig:WsWuH0}, i.e.,

\[ W^u(P_5)=\{(r(s),\alpha_0,\alpha_0)= \{ ( r_0 e^{\mu_1s} ,\alpha_0,\alpha_0) , r_0>0\}.\]
Analogously,
\[ W^s(P_6)=\{ (  r_0 e^{-\mu_1s},\alpha_0,\alpha_0+\pi),  r_0>0\}.\]

Here $\mu_1=\sin\alpha_0\cos\alpha_0=  \sqrt{(q - 1) (p - 1)}/( p + q - 2) >0$ is an eigenvalue of $D\tilde Y(P_5)$. The projection of both is the  half straight line $y=\tan\alpha_0 x$, $x> 0$ and  it is shown in Figure \ref{fig:main-theorem}.

 \begin{figure}[H]
	\begin{center}
  \def\svgwidth{1.0\textwidth}
 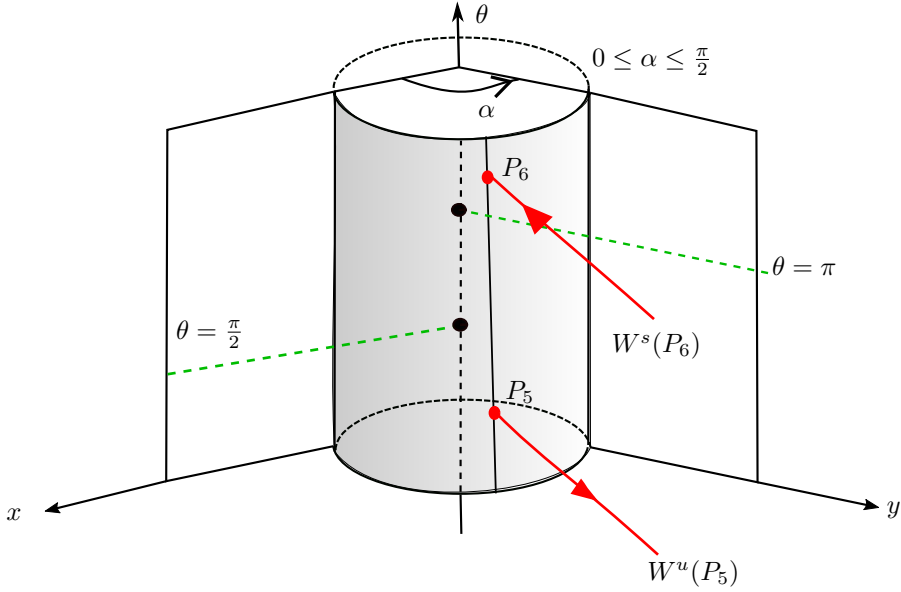
 	 	\vskip .5cm
		\caption{Blowing-up and invariant manifolds when $H=0$. In this case the invariant manifolds $W^u(P_5)$ and $W^s(P_6)$ are straight lines.}
		\label{fig:WsWuH0}
	\end{center}
\end{figure}

 Thus, we have that $\pi(W^u(P_5))$ is parametrized by
 \[ \Gamma^u(s)=e^{\mu_1s} (\cos\alpha_0 ,  \sin\alpha_0)\]
 and 
  $\pi(W^s(P_6))$ is parametrized by
 \[ \Gamma^s(s)=e^{-\mu_1s} \left(\cos\alpha_0 ,  \sin\alpha_0\right).\]

Therefore, when $H=0$ the projections of the integral curves of $Y$ in the plane $xy$ passing through the origin are the half lines $y=\tan\alpha_0 x$.

In order to conclude the proof, notice that the line $y=\tan\alpha_0 x$ can be expressed in the parametric form
\[
\left\{\begin{aligned}
x(t)&=\cos\alpha_0 t=t \sqrt{q-1}/\sqrt{p+q-2},\\
y(t)&=\sin\alpha_0 t=t \sqrt{p-1}/\sqrt{p+q-2}.
\end{aligned}\right.
\] 
 Computing the generated hypersurface invariant by $O(p)\times O(q),$ we obtain
that it is the cone
 given by
 \[ C=\{ (U,V)\in \mathbb{R}^p\times\mathbb{R}^q ;(q-1)|V|^2=(p-1)|U|^2\}. \]
 The evaluation  of the principal curvatures is direct from equation \eqref{eq:principal_curvatures}.
\end{proof}

\subsection*{Acknowledgements}

Hil\'ario Alencar, Ronaldo Garcia and Greg\'orio Silva Neto were Partially supported by the National Council for Scientific and Technological Development--CNPq of Brazil. Greg\'orio Silva Neto was also partially supported by the Alagoas Research Foundation of Brazil. 

 \bibliographystyle{plain}

\vskip .5cm
\newpage
\author{\noindent  Hilário Alencar\\
Instituto de Matemática\\
Universidade Federal de Alagoas,\\
57072--970, Maceió, Brazil}

\noindent{e-mail: hilario@mat.ufal.br}

 \vskip 0.5cm

\author{\noindent Ronaldo Garcia\\
Instituto de Matem\'atica e Estat\'{\i}stica \\
Universidade Federal de Goi\'as,\\
74690--900, Campus Samambaia\\
Goi\^ania, Brazil}

 \noindent{e-mail: ragarcia@ufg.br}

 \vskip 0.5cm

\author{\noindent  Gregório Silva Neto\\
Instituto de Matemática\\
Universidade Federal de Alagoas\\
57072--970,
Maceió, Brazil}

\noindent{e-mail: gregorio@im.ufal.br}

 \end{document}